\documentclass[graybox]{svmult}

\usepackage{mathptmx}       
\usepackage{helvet}         
\usepackage{courier}        
\usepackage{type1cm}        
\usepackage{makeidx}         
\usepackage{graphicx}        
\usepackage{multicol}        
\usepackage[bottom]{footmisc}

\usepackage{tikz}
\usepackage{pgfplots}
\usepgflibrary{shapes.geometric} 
\usepgflibrary[shapes.geometric] 
\usetikzlibrary{shapes.geometric} 
\usetikzlibrary[shapes.geometric] 
\usepackage{mathptmx,amssymb,amsmath,amscd}
\usetikzlibrary{arrows,shapes,backgrounds}

\newcommand{\Pin}{\mathop{\mathrm{Pin}}}
\newcommand{\Spin}{\mathop{\mathrm{Spin}}}
\newcommand{\Cl}{\mathop{\mathrm{Cl}}}
\newcommand{\Aut}{\mathop{\mathrm{Aut}}}

\makeindex             

\begin{document}

\title*{A 3D spinorial view of 4D exceptional phenomena}
\author{Pierre-Philippe Dechant}
\institute{Pierre-Philippe Dechant \at Department of Mathematics \\ University of York \\  Heslington, York, YO10 5GG \\ United Kingdom, \email{ppd22@cantab.net}
}
%
%
\maketitle


\abstract{
We discuss a Clifford algebra framework for discrete symmetry groups (such as
reflection, Coxeter, conformal and modular groups), leading to a surprising
number of new results. Clifford algebras allow for a particularly simple
description of reflections via `sandwiching'. This extends to a description of
orthogonal transformations in general by means of `sandwiching' with Clifford
algebra multivectors, since all orthogonal transformations can be written as products
of reflections by the Cartan-Dieudonn\'e theorem.
We begin by viewing the largest non-crystallographic reflection/Coxeter group $H_4$ as a group of rotations in two different ways -- firstly via a folding from the largest exceptional group $E_8$, and secondly by induction from the icosahedral group $H_3$ via Clifford spinors. 
We then generalise the second way by
presenting a construction of a 4D root system from any given 3D one.
This affords a new -- spinorial -- perspective on 4D phenomena, in particular as the induced root systems are precisely the exceptional ones in 4D, and their unusual automorphism groups are easily explained in the spinorial picture; we discuss the wider context of Platonic solids, Arnold's trinities and the McKay correspondence.
The multivector groups can be used to perform concrete group-theoretic calculations, e.g. those for $H_3$ and $E_8$, and we discuss how various representations can also be constructed in this Clifford framework; in particular,  representations of quaternionic type arise very naturally.}

\section{Introduction}\label{sec_intro}

Reflections are the building blocks for a large class of discrete symmetries that are of interest in both pure and applied mathematics. Coxeter groups, root systems and polytopes are intimately related to Lie groups and algebras, as well as to the geometry of various dimensions. The geometry of three dimensions has manifold obvious implications for physics, chemistry and biology; in particular, in our recent work we were interested in the role of icosahedral symmetry to virus structure, fullerenes and quasicrystals \cite{DechantTwarockBoehm2011H3aff, Dechant:eo5029, DechantTwarockBoehm2011E8A4}. Lie groups and algebras, as well as their root lattices and Coxeter/Weyl groups are also ubiquitous in high energy physics
\cite{Gross1985HE, DamourHenneauxNicolai2002E10, Koca2001H4E8,HennauxPersson2008SpacelikeSingularitiesAndHiddenSymmetriesofGravity}.  
As we shall see later, even conformal and modular groups fall into this category of discrete groups; the relevant areas in mathematical physics include conformal field theory \cite{Schellekens1996CFT} and Moonshine phenomena \cite{gannon2006moonshine, eguchi2011notes,eguchi2007liouville,taormina2013twist}. 

In this reflection framework one always has an inner product on the $n$-dimensional vector space in question; one can thus in fact always construct the corresponding $2^n$-dimensional Clifford algebra, which contains the original $n$-dimensional vector space as a subspace. This is particularly useful, as Clifford algebra allows a dramatic simplification when it comes to handling reflections.
We have explored such a Clifford algebra framework in
\cite{Dechant2012CoxGA,  Dechant2013Platonic, Dechant2012Induction, Dechant2015ICCA} from the pure mathematics perspective; this has led to a number of conceptual and computational simplifications, as well as some very profound results on the nature of four-dimensional (4D) geometry and its interplay with the geometry of three dimensions (3D), in particular that of rotations. Here we present an account of this work tailored to finite group theorists, exploring various pure connections, as well as presenting new work on group and representation theory. 

 This paper is organised as follows. 
Section \ref{sec_Cox} introduces root systems, reflection and Coxeter groups, and their graphical representations. 
In Section \ref{sec_versor} we present some Clifford algebra background, in particular the unique reflection prescription and the resulting versor formalism via the Cartan-Dieudonn\'e theorem.
We discuss two ways of viewing the reflection group $H_4$ as a group of rotations -- 
in Section \ref{sec_spin} we discuss $H_4$ as a subgroup of $E_8$;
in Section \ref{sec_bin} we consider the group of $120$ multivectors generated by the simple root vectors of $H_3$ via multiplication in the Clifford algebra. This is the binary icosahedral group $2I$, and its multivector components are exactly the roots of $H_4$. 
We then generalise this observation: 
this yields a remarkable theorem which induces  a 4D root system from every 3D root system in a constructive way (Section \ref{sec_pin}). We discuss this construction -- which uses Clifford spinors -- and its relation to the 4D Platonic solids  along with their peculiar symmetries, as well as the wider context of Arnold's trinities and the McKay correspondence,  which this construction puts into a wider framework. 
In Section \ref{sec_CGA} we revisit some of the earlier group-theoretic calculations and show how Clifford multivectors can be used to construct more conventional matrix representations; in particular, certain  representations of quaternionic type of a number of specific groups arise uniformly in this construction. 
We conclude with a summary and possible further work in Section \ref{sec_concl}.

\section{Root systems and reflection groups}\label{sec_Cox}

In this section, we introduce reflection/Coxeter groups as generated by their root systems.
Let $V$ be an $n$-dimensional Euclidean vector space endowed with a positive
definite bilinear form $(\cdot |\cdot )$.
A \emph{root system} is a collection $\Phi$ of non-zero vectors (called root vectors)
satisfying the following
two axioms:
\begin{enumerate}
\item $\Phi$ only contains a root $\alpha$ and its negative, but no other scalar multiples: $\Phi \cap \mathbb{R}\alpha=\{-\alpha, \alpha\}\,\,\,\,\text{ for every }\,\, \alpha \in \Phi$. 
\item $\Phi$ is invariant under all reflections corresponding to root vectors. That is, if
$s_\alpha$ is the reflection of $V$ in the hyperplane with normal $\alpha$, we require that 
 $s_\alpha\Phi=\Phi \,\,\,\text{ for every }\,\, \alpha\in\Phi$. 
\end{enumerate}
For a crystallographic root system, a subset $\Delta$ of $\Phi$, called \emph{simple roots} $\alpha_1, \dots, \alpha_n$, is sufficient to express every element of $\Phi$ via $\mathbb{Z}$-linear combinations with coefficients of the same sign. 
$\Phi$ is therefore  completely characterised by this basis of simple roots. In the case of the non-crystallographic root systems $H_2$, $H_3$ and $H_4$, the same holds for the extended integer ring $\mathbb{Z}[\tau]=\lbrace a+\tau b| a,b \in \mathbb{Z}\rbrace$, where $\tau$ is   the golden ratio $\tau=\frac{1}{2}(1+\sqrt{5})=2\cos{\frac{\pi}{5}}$, and $\sigma$ is its Galois conjugate $\sigma=\frac{1}{2}(1-\sqrt{5})$ (the two solutions to the quadratic equation $x^2=x+1$).  
 For the crystallographic root systems, the classification in terms of Dynkin diagrams essentially follows the one familiar from Lie groups and Lie algebras, as their Weyl groups are  the crystallographic Coxeter groups. A mild generalisation to so-called Coxeter-Dynkin diagrams is necessary for the non-crystallographic root systems, 
 where nodes correspond to simple roots, orthogonal roots are not connected, roots at $\frac{\pi}{3}$ have a simple link, and other angles $\frac{\pi}{m}$ have a link with a label $m$. 
The  \emph{Cartan matrix} of a set of simple roots $\alpha_i\in\Delta$ is defined as the matrix
	$A_{ij}=2{(\alpha_i\vert \alpha_j)}/{(\alpha_j\vert \alpha_j)}$.
For instance, the root system of the icosahedral group $H_3$ has one link labelled by $5$ (via the above relation $\tau=2\cos{\frac{\pi}{5}}$), as does its four-dimensional analogue $H_4$. A plethora of examples of diagrams is presented later, in Table \ref{tab:4} in Section \ref{sec_bin}.

The reflections in the second axiom of the root system generate a reflection group. A Coxeter group is a mathematical abstraction of the concept of a reflection group via  involutory  generators (i.e. their square is the identity, which captures the idea of a reflection), subject to mixed relations that represent $m$-fold rotations (since two successive reflections generate a rotation in the plane spanned by the two roots).
A \emph{Coxeter group} is a group generated by a set of involutory generators $s_i, s_j \in S$ subject to relations of the form $(s_is_j)^{m_{ij}}=1$ with $m_{ij}=m_{ji}\ge 2$ for $i\ne j$.  
The  finite Coxeter groups have a geometric representation where the involutions are realised as reflections at hyperplanes through the origin in a Euclidean vector space $V$, i.e. they are essentially  just the classical reflection groups. In particular, then the abstract generator $s_i$ corresponds to the simple {reflection}
$s_i: \lambda\rightarrow s_i(\lambda)=\lambda - 2\frac{(\lambda|\alpha_i)}{(\alpha_i|\alpha_i)}\alpha_i$
 in the hyperplane perpendicular to the  simple {root } $\alpha_i$.
The action of the Coxeter group is  to permute these root vectors, and its  structure is thus encoded in the collection  $\Phi\in V$ of all such roots, which in turn form a root system.

We now move onto the Clifford algebra framework, which affords a uniquely simple prescription for performing reflections (and thus any orthogonal transformation) in spaces of any dimension and signature. For any root system, the quadratic form mentioned in the definition always allows one to enlarge the $n$-dimensional vector space $V$ to the corresponding $2^n$-dimensional Clifford algebra. The Clifford algebra  is therefore a very natural object to consider in this context, as its unified structure simplifies many problems both conceptually and computationally (as we shall see in the next section), rather than applying the linear structure of the space and the inner product separately.

\section{Clifford versor framework}\label{sec_versor}

Clifford algebra can be viewed as a deformation of the (perhaps more familiar) exterior algebra by a quadratic form -- though we do not necessarily advocate this point of view; they are isomorphic as vector spaces, but not as algebras, and Clifford algebra is in fact much richer due to the invertibility of the algebra product, as we shall see. 
The \emph{geometric product} of Geometric/Clifford Algebra is defined by $xy=x\cdot y+x \wedge y$ -- where the scalar product (given by the symmetric bilinear form) is the symmetric part $x\cdot y=(x|y)=\frac{1}{2}(xy+yx)$ and the exterior product the antisymmetric part $x\wedge y=\frac{1}{2}(xy-yx)$   \cite{Hestenes1966STA, HestenesSobczyk1984, Hestenes1990NewFound,LasenbyDoran2003GeometricAlgebra}.
It provides a very compact and efficient way of handling reflections in any number of dimensions, and thus by the \emph{Cartan-Dieudonn\'e theorem} in fact of any orthogonal transformation. For a unit vector $\alpha$, the two terms in the above formula for a reflection of a vector $v$ in the hyperplane orthogonal to $\alpha$ simplify to the double-sided (`sandwiching') action of $\alpha$ via the geometric product
	\begin{equation}\label{in2refl}
	  v\rightarrow s_\alpha v=v'=v-2(v|\alpha)\alpha=v-2\frac{1}{2}(v\alpha+\alpha v)\alpha=v-v\alpha^2-\alpha v\alpha=-\alpha v \alpha.
	\end{equation}
This  prescription for reflecting vectors in hyperplanes is remarkably compact (note that $\alpha$ and $-\alpha$ encode the same reflection and thus provide a double cover). Via the Cartan-Dieudonn\'e theorem, any orthogonal transformation can be written as the product of reflections,  and thus by performing consecutive reflections each given via `sandwiching', one is led to define a versor as
 a Clifford multivector $A=a_1a_2\dots a_k$, that is the product of $k$ unit vectors $a_i$  \cite{Hestenes1990NewFound}.  Versors form a multiplicative group called the versor/pinor group $\Pin$ under the single-sided multiplication with the geometric product, with inverses given by $\tilde{A}A=A\tilde{A}=\pm 1$, where the tilde denotes the reversal of the order of the constituent vectors $\tilde{A}=a_k\dots a_2a_1$, and  the $\pm$-sign defines its parity.
Every orthogonal transformation $\underbar{A}$ of a vector $v$ can thus be expressed by means of unit versors/pinors via
\begin{equation}\label{in2versor}
\underbar{A}: v\rightarrow  v'=\underbar{A}(v)=\pm{\tilde{A}vA}.
\end{equation}
Unit versors are double-covers of the respective orthogonal transformation, as $A$ and $-A$ encode the same transformation. Even versors $R$, that is, products of an even number of vectors, are called spinors or rotors. They form a subgroup of the $\Pin$ group and constitute a double cover of the special orthogonal group, called the $\Spin$ group.
Clifford algebra therefore provides a particularly natural and simple construction of the $\Spin$ groups. Thus the remarkably simple construction of the binary polyhedral groups in Section \ref{sec_pin} is not at all surprising from a Clifford point of view, but appears to be unknown in the Coxeter community, and ultimately leads to the novel result of the spinor induction theorem of (exceptional) 4D root systems in Section \ref{sec_pin}.

\begin{table}
\caption{Versor framework for a unified treatment of the chiral, full,  binary and pinor polyhedral groups}
\label{tab:1}       
%
%
\begin{tabular}{p{2.6cm}p{3.3cm}p{4.9cm}}
\hline\noalign{\smallskip}
Continuous group &Discrete subgroup & Multivector action  \\
\noalign{\smallskip}\svhline\noalign{\smallskip}
$SO(n)$&rotational/chiral & $x\rightarrow \tilde{R}xR$\\
$O(n)$&reflection/full & $x\rightarrow \pm\tilde{A}xA$\\
$\Spin(n)$&binary  & spinors $R$ under $(R_1,R_2)\rightarrow R_1R_2$\\
$\Pin(n)$& pinor & pinors $A$ under $(A_1,A_2)\rightarrow A_1A_2$\\
\noalign{\smallskip}\hline\noalign{\smallskip}
\end{tabular}
\end{table}

The versor realisation of the orthogonal group is much simpler than conventional matrix approaches. 
Table \ref{tab:1} summarises the various action mechanisms of multivectors: a rotation (e.g. the continuous group $SO(3)$ or the discrete subgroup, the chiral icosahedral group $I=A_5$) is given by double-sided action of a spinor $R$, whilst these spinors themselves form a group under single-sided action/multiplication (e.g. the continuous group $\Spin(3)\sim SU(2)$ or the discrete subgroup, the binary icosahedral group $2I$).
Likewise, a reflection (continuous $O(3)$ or the discrete subgroup, the full icosahedral group the Coxeter group $H_3$) corresponds to sandwiching with the versor $A$, whilst the versors single-sidedly form a multiplicative group (the $\Pin(3)$ group or the discrete analogue, the double cover of $H_3$, which we denote $\Pin(H_3)$). In the conformal geometric algebra setup one uses the fact that the conformal group   $C(p,q)$ is homomorphic to $SO(p+1,q+1)$ to treat translations as well as rotations in a unified versor framework  \cite{HestenesSobczyk1984, LasenbyDoran2003GeometricAlgebra,Dechant2011Thesis, Dechant2012AGACSE, Dechant2015ICCA}. \cite{Dechant2012AGACSE, Dechant2015ICCA} also discuss reflections, inversions, translations and modular transformations in this way. 

\begin{example}
The Clifford/Geometric algebra of three dimensions $\Cl(3)$ is spanned by three orthogonal -- and thus anticommuting -- unit vectors $e_1$, $e_2$ and $e_3$. It also contains the three bivectors $e_1e_2$, $e_2e_3$ and $e_3e_1$ that all square to $-1$, as well as the  highest grade object $e_1e_2e_3$   (trivector and pseudoscalar), which also squares to $-1$. Therefore, in Clifford algebra various geometric objects arise that provide imaginary structures; however, there can be different ones and they can have non-trivial commutation relations with the rest of the algebra. 
\begin{equation}\label{in2PA}
  \underbrace{\{1\}}_{\text{1 scalar}} \,\,\ \,\,\,\underbrace{\{e_1, e_2, e_3\}}_{\text{3 vectors}} \,\,\, \,\,\, \underbrace{\{e_1e_2=Ie_3, e_2e_3=Ie_1, e_3e_1=Ie_2\}}_{\text{3 bivectors}} \,\,\, \,\,\, \underbrace{\{I\equiv e_1e_2e_3\}}_{\text{1 trivector}}.
\end{equation}
\end{example}

\section{$H_4$ as a rotation rather than reflection group I:  from $E_8$}\label{sec_spin}

The largest exceptional Coxeter group $E_8$ and the largest non-crystallographic Coxeter group $H_4$ are closely related. 
This connection between $E_8$ and $H_4$ can be shown via Coxeter-Dynkin diagram foldings on the level of Coxeter groups \cite{Shcherbak:1988} or as a projection relating the root systems \cite{MoodyPatera:1993b, DechantTwarockBoehm2011E8A4}. 
On the level of the root system this is due to the existence of a projection which maps the $240$ roots of $E_8$ onto the $120$ roots of $H_4$ and their $\tau$-multiples. 
We now consider the Dynkin diagram folding picture in more detail.

\begin{figure}
	\begin{center}
	
\begin{tikzpicture}[scale=0.5,
    knoten/.style={        circle,      inner sep=.1cm,        draw}
   ]
  
  \node at (1,1.6) (knoten1) [knoten,  color=white!0!black] {};
  \node at (3,1.6) (knoten2) [knoten,  color=white!0!black] {};
  \node at (5,1.6) (knoten3) [knoten,  color=white!0!black] {};
  \node at (7,1.6) (knoten4) [knoten,  color=white!0!black] {};

  \node at (1,.6) (knoten5) [knoten,  color=white!0!black] {};
  \node at (3,.6) (knoten6) [knoten,  color=white!0!black] {};
  \node at (5,.6) (knoten7) [knoten,  color=white!0!black] {};
  \node at (7,.6) (knoten8) [knoten,  color=white!0!black] {};

\node at (1,2.2)  (alpha1) {$\alpha_1$};
\node at (3,2.2)  (alpha2) {$\alpha_2$};
\node at (5,2.2)  (alpha3) {$\alpha_3$};
\node at (7,2.2)  (alpha4) {$\alpha_4$};
\node at (1,0)  (alpha7) {$\alpha_7$};
\node at (3,0)  (alpha6) {$\alpha_6$};
\node at (5,0)  (alpha5) {$\alpha_5$};
\node at (7,0)  (alpha8) {$\alpha_8$};
\node at (9,1.1)  (ra) {$\Rightarrow$};

  \path  (knoten1) edge (knoten2);
  \path  (knoten2) edge (knoten3);
  \path  (knoten3) edge (knoten4);
  \path  (knoten5) edge (knoten6);
  \path  (knoten6) edge (knoten7);
  \path  (knoten7) edge (knoten8);
  \path  (knoten4) edge (knoten7);

		  \node at (11,1.0) (knoten1) [knoten,  color=white!0!black] {};
		  \node at (13,1.0) (knoten2) [knoten,  color=white!0!black] {};
		  \node at (15,1.0) (knoten3) [knoten,  color=white!0!black] {};
		  \node at (17,1.0) (knoten4) [knoten,  color=white!0!black] {};

		\node at (11,0.4)  (a1) {$a_1$};
		\node at (13,0.4)  (a2) {$a_2$};
		\node at (15,0.4)  (a3) {$a_3$};
		\node at (16,1.45)  (tau) {$5$};
		\node at (17,0.4)  (a4) {$a_4$};

		  \path  (knoten1) edge (knoten2);
		  \path  (knoten2) edge (knoten3);
		  \path  (knoten3) edge (knoten4);

		\end{tikzpicture}
  \caption[$E_8$]{Coxeter-Dynkin diagram folding from $E_8$ to $H_4$. 
Note that deleting nodes $\alpha_1$ and $\alpha_7$ yields corresponding results for $D_6\rightarrow H_3$, and likewise for  $A_4\rightarrow H_2$ by further removing $\alpha_2$ and $\alpha_6$.}
\label{figE8}
\end{center}
\end{figure}
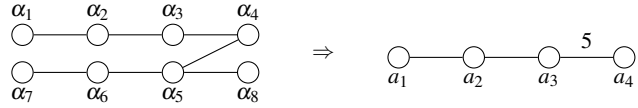

We take the simple roots $\alpha_1$ to  $\alpha_8$ of $E_8$ as shown in  Fig. \ref{figE8}, and consider the Clifford algebra in 8D with the usual Euclidean metric. The simple reflections corresponding to the simple roots are thus just given via sandwiching 
 $s_\alpha v=-\alpha v \alpha$ as in Eqn. (\ref{in2refl}). The Coxeter element $w$ is defined as the product of all these eight simple reflections, and in Clifford algebra it is therefore simply given by the corresponding (s)pinor $W=\alpha_1\dots \alpha_8$ acting via sandwiching. Its order, the Coxeter number $h$ (that is, the smallest $h$ such that $W^h=\pm 1$), is $30$ for $E_8$. 

As illustrated in Figure \ref{figE8}, one can define certain combinations of pairs of reflections (corresponding to roots on top of each other in the Dynkin diagram folding), e.g. $s_{a_1}=s_{\alpha_1} s_{\alpha_7}$ etc, and in a Clifford algebra sandwiching way these are given by the products of root vectors ${a_1}={\alpha_1} {\alpha_7}$,   ${a_2}={\alpha_2} {\alpha_6}$, ${a_3}={\alpha_3} {\alpha_5}$ and   ${a_4}={\alpha_4} {\alpha_8}$ (which is essentially a partial folding of the usual alternating folding used in the construction of the Coxeter plane with symmetry group $I_2(h)$). It is easy to show that the subgroup with the generators $s_{a_i}$  in fact satisfies the relations of an  $H_4$ Coxeter group \cite{Bourbaki1981Lie, Shcherbak:1988}: because of the Coxeter relations for $E_8$ and the orthogonality of the combined pair the combinations $s_a$ are easily seen to be involutions, and the 3-fold relations are similarly obvious from the  Coxeter relations; only for the 5-fold relation does one have to perform an explicit calculation in terms of the reflections with respect to the root vectors. This is thus particularly easy by multiplying together vectors in the Clifford algebra, rather than by concatenating two reflection formulas of the usual type -- despite   containing just two terms, concatenation gets convoluted quickly, which is not the case in the multiplication of multivectors. 

Since the combinations $s_a$ are pairs of reflections, they are obviously rotations in the eight-dimensional space, so this $H_4$ group acts as rotations in the full space, but as a reflection group in a 4D subspace. The $H_4$ Coxeter element is given by multiplying together the four combinations $a_i$ -- its Coxeter versor is therefore trivially seen to be the same as that of $E_8$ (up to sign, since orthogonal vectors anticommute) and the Coxeter number of $H_4$ is thus the same as that of $E_8$, $30$. The projection of the $E_8$ root system onto the Coxeter plane consists of two copies of the projection of $H_4$ into the Coxeter plane, with a relative factor of $\tau$.

\section{$H_4$ as a rotation rather than reflection group II: from $H_3$}\label{sec_bin}

In the previous section we have considered certain generators and multivectors in the algebra. We now consider the whole (s)pinor group generated by the simple reflections of $H_3$. The simple roots are taken as $\alpha_1=e_2$, $\alpha_2=-\frac{1}{2}((\tau -1)e_1+e_2+\tau e_3)$, and $\alpha_3=e_3$. Under free multiplication, these generate a group with $240$ elements (pinors), and the even subgroup consists of $120$ elements (spinors), for instance of the form $\alpha_1\alpha_2=-\frac{1}{2}(1-(\tau -1)e_1e_2+\tau e_2e_3)$ and $\alpha_2\alpha_3=-\frac{1}{2}(\tau-(\tau -1)e_3e_1+e_2e_3)$. These are the double covers of $I=A_5$ and $H_3=A_5\times \mathbb{Z}_2$, respectively. With these groups of multivectors one can perform standard group theory calculations, such as finding inverses and   conjugacy classes. The spinors have four components ($1$, $e_1e_2$, $e_2e_3$, $e_3e_1$); by taking the components of these $120$ spinors as a set of vectors in 4D one obtains the $120$ roots in the  $H_4$ root system. This is very surprising from a Coxeter perspective, as one usually thinks of $H_3$ as a subgroup of $H_4$, and therefore of $H_4$ as more `fundamental'; however,  one now sees that $H_4$ does not in fact contain any structure that is not already contained in $H_3$, and can therefore think of $H_3$ as more fundamental. We will present a general and uniform construction explaining and systematising this fact in the next section. 

From a Clifford perspective it is not surprising to find this group of $120$ spinors, which is the binary icosahedral group, since as we have seen Clifford algebra provides a simple construction of the Spin groups. This spinor group, the binary icosahedral group $2I$, has $120$ elements and $9$ conjugacy classes, and calculations in the Clifford algebra are very straightforward; standard approaches would have to somehow move from $SO(3)$ rotation matrices to $SU(2)$ matrices for the binary group -- here both are treated in the same framework. 
The fact that the rotational icosahedral group $I$ (given by double-sided action of spinors $R$ as $\tilde{R}xR$) has five conjugacy classes  and it being of order $60$ imply that this group has five irreducible representations of dimensions  $1$, $3$, $\bar{3}$, $4$ and $5$ (since the sum of the dimensions squared gives the order of the group $\sum d_i^2=|G|$). The nine conjugacy classes of the binary icosahedral group $2I$ of order $120$ (given by the spinors $R$ under algebra multiplication) imply that this acquires a further four irreducible spinorial representations $2_s$, $2_s'$, ${4_s}$ and ${6_s}$.

\begin{figure}
	\begin{center}
\begin{tikzpicture}[scale=0.5,
knoten/.style={        circle,      inner sep=.1cm,        draw}
]
\node at (-1,.7) (knoten0) [knoten,  color=white!0!black] {};
\node at  (1,.7) (knoten1) [knoten,  color=white!0!black] {};
\node at  (3,.7) (knoten2) [knoten,  color=white!0!black] {};
\node at  (5,.7) (knoten3) [knoten,  color=white!0!black] {};
\node at  (7,.7) (knoten4) [knoten,  color=white!0!black] {};
\node at  (9,.7) (knoten6) [knoten,  color=white!0!black] {};
\node at  (11,.7) (knoten7) [knoten,  color=white!0!black] {};
\node at  (13,.7) (knoten8) [knoten,  color=white!0!black] {};
\node at  (9,2.7) (knoten9) [knoten,  color=white!0!black] {};

\node at  (-1,0) (alpha0)  {$1$};
\node at  (1,0)  (alpha1) {$2_s$};
\node at  (3,0)  (alpha2) {$3$};
\node at  (5,0)  (alpha3) {$4_s$};
\node at  (7,0)  (alpha4) {$5$};
\node at  (9,0)  (alpha5) {$6_s$};
\node at  (11,0)  (alpha6) {$4$};

\node at (9.7,2.8)  (alpha7) {$\bar{3}$};
\node at (13,0) (alpha8) {$2_s'$};

\path  (knoten0) edge (knoten1);
\path  (knoten1) edge (knoten2);
\path  (knoten2) edge (knoten3);
\path  (knoten3) edge (knoten4);
\path  (knoten4) edge (knoten6);
\path  (knoten6) edge (knoten9);
\path  (knoten6) edge (knoten7);
\path  (knoten7) edge (knoten8);

\end{tikzpicture} 
\end{center}
\caption[$E_6^+$]{The graph depicting the tensor product structure of the binary icosahedral group $2I$ is the same as the Dynkin diagram for the  affine extension of $E_8$, $E_8^+$. }
\label{figE6aff}
\end{figure}
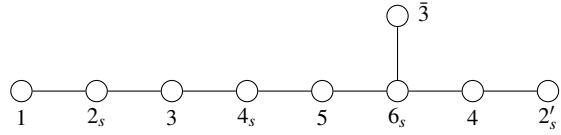

The binary icosahedral group has a curious connection with the affine Lie algebra $E_8^+$ (which also applies to the other binary polyhedral groups and the affine Lie algebras of $ADE$-type) via the so-called McKay correspondence \cite{Mckay1980graphs}, which is twofold. 
 First, we may define a graph by assigning a node to each irreducible representation of the binary icosahedral group with the following rule for connecting nodes: each node corresponding to a certain irreducible representation is connected to the nodes corresponding to those irreducible representations that are contained in its tensor product with the irrep $2_s$. For instance, tensoring the trivial representation $1$ with $2_s$ trivially gives $2_s$ and thus the only link  $1$ has is with $2_s$; $2_s\otimes 2_s=1+3$, such that $2_s$ is connected to $1$ and $3$, etc. The graph that is built up in this way is precisely the Dynkin diagram of affine $E_8$, as shown in Figure \ref{figE6aff}. The second connection is the following observation: the Coxeter element is the product of all the simple reflections $\alpha_1\dots \alpha_8$ and its order, the Coxeter number $h$, is $30$ for $E_8$. This also happens to be the sum of the dimensions of the irreducible representations of $2I$. This extends to all other cases of polyhedral groups and $ADE$-type affine Lie algebras, as shown in the second and third columns in Table \ref{tab:4} and in Figure \ref{figMcKay}.

\section{The general construction: spinor induction and the 4D Platonic Solids, Trinities and McKay correspondence}\label{sec_pin}

In this section we systematise the above observation. Starting with any 3D root system, we present a construction that yields a 4D root system; the intermediate steps involve Clifford spinor techniques. We begin by an auxiliary result.
\begin{proposition}[$O(4)$-structure of spinors]\label{HGA_O4}
The space of $\Cl(3)$-spinors $R=a_0+a_1Ie_1+a_2Ie_2+a_3Ie_3$ can be endowed with a \emph{4D Euclidean norm} $|R|^2=R\tilde{R}=a_0^2+a_1^2+a_2^2+a_3^2$ induced by the  \emph{inner product} $(R_1,R_2)=\frac{1}{2}(R_1\tilde{R}_2+R_2\tilde{R}_1)$ between  two spinors $R_1$ and $R_2$. 
\end{proposition}
This allows one to  reinterpret the group of 3D spinors generated from a 3D root system as a set of 4D vectors, which in fact can be shown to  satisfy the axioms of a root system as given above. 
\begin{theorem}[Induction Theorem]\label{HGA_4Drootsys}
Any 3D root system gives rise to a spinor group $G$ which induces a root system in 4D.
\end{theorem}
\begin{proof}
It is sufficient to check the two axioms for the root system $\Phi$ consisting of the set of 4D vectors given by the 3D spinor group:
\begin{enumerate}
\item By construction, $\Phi$ contains the negative of a root $R$ since spinors provide a double cover of rotations, i.e. if $R$ is in a spinor group $G$, then so is $-R$ , but no other scalar multiples (normalisation to unity). 
\item $\Phi$ is invariant under all reflections with respect to the inner product $(R_1,R_2)$ in Proposition \ref{HGA_O4} since $R_2'=R_2-2(R_1, R_2)/(R_1, {R}_1) R_1=-R_1\tilde{R}_2R_1\in G$ for $R_1, R_2 \in G$ by the closure property of the group $G$ (in particular $-R$ and $\tilde{R}$ are in $G$ if $R$ is). 
\end{enumerate}
\end{proof}

Since the number of irreducible 3D root systems is limited to $(A_3, B_3, H_3)$, this yields a definite list of induced root systems in 4D; this turns out to be  $(D_4, F_4, H_4)$, which are exactly the exceptional root systems in 4D. 
In fact, the two triples may be regarded as trinities in Arnold's
sense, originally              applying to the trinity $(\mathbb{R},\mathbb{C},\mathbb{H})$ and later extended to  projective spaces, Lie algebras, spheres, Hopf fibrations etc \cite{Arnold1999symplectization,Arnold2000AMS}. Arnold's original link between our trinities $(A_3, B_3, H_3)$ and $(D_4, F_4, H_4)$ was extremely convoluted, and our construction presents a novel  direct link between the two. 

These root systems are intimately linked to the Platonic solids -- there are 5 in three dimensions and 6 in four dimensions: $A_3$ is the root system relevant to the tetrahedron, $B_3$  generates the symmetries of the cube and octahedron (which are dual under the exchange of faces and vertices), and $H_3$ describes the symmetries of the dual pair icosahedron and dodecahedron (the rotational subgroup is denoted by $I= A_5$). 

Likewise, the 4D Coxeter groups describe the symmetries of the 4D Platonic solids, but this time the connection is more immediate since the root systems are actually Platonic solids themselves: $D_4$ is the $24$-cell (self-dual), an analogue of the tetrahedron, which is also related to the $F_4$ root system, and the $H_4$ root system is the Platonic solid the $600$-cell. Its dual, the $120$-cell obviously has the same symmetry. The root system $A_1^3$ generates the root system $A_1^4$, which constitutes the vertices of the Platonic solid $16$-cell, with the $8$-cell as its dual.  There is thus an abundance of root systems in 4D that are related to the Platonic solids, and in fact the only one not equal or dual to a root system is the $5$-cell with symmetry group $A_4$ -- which of course could not be a root system, as it has an odd number ($5$) of vertices. This abundance of root systems in 4D can in some sense be thought of as due to the accidentalness of this construction.  In particular, the induced root systems are precisely the exceptional (i.e. they do not have counterparts in other dimensions) root systems in 4D: $D_4$ has the  triality symmetry (permutation symmetry of the three legs in the diagram) that is exceptional in 4D, $F_4$ is the only $F$-type root system, and $H_4$ is the largest non-crystallographic root system. In contrast, in arbitrary dimensions there are only   $A_n$ ($n$-simplex), and $B_n$  ($n$-hypercube and $n$-hyperoctahedron).

\begin{table}
\caption{Clifford spinor construction and McKay correspondence: connections between 3D, 4D and $ADE$-type root systems. $|\Phi|$, $\sum d_i$ and $h$ are $6$, $12$, $18$ and $30$, respectively.}
\label{tab:4}       
%
%
\begin{tabular}{p{0.4cm}p{3cm}p{0.4cm}p{3cm}p{5.0cm}}
\hline\noalign{\smallskip}
&3D root system &&4D root system/binary polyhedral group & affine Lie algebra  \\
\noalign{\smallskip}\svhline\noalign{\smallskip}
$A_1^3$ &\includegraphics[width=2cm]{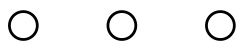}&$A_1^4$ &\includegraphics[width=2cm]{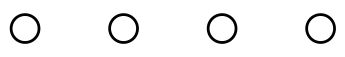}& $D_4^+$ \includegraphics[width=2cm]{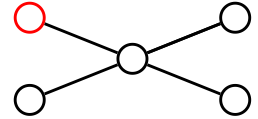}\\
\hline\noalign{\smallskip}
$A_3$ &\includegraphics[width=2cm]{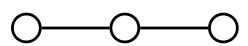}&$D_4$ &\includegraphics[width=2cm]{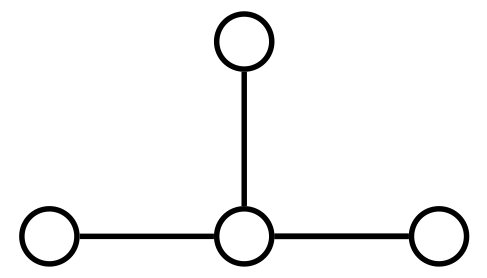}& $E_6^+$ \includegraphics[width=2cm]{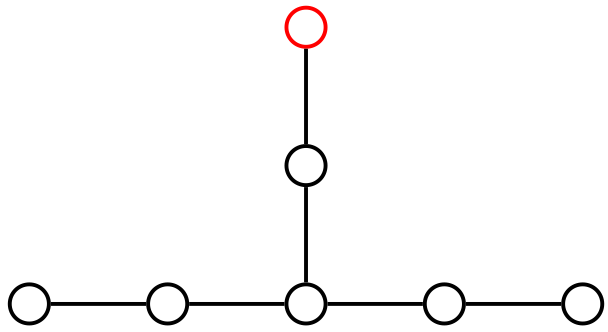}\\
\hline\noalign{\smallskip}
$B_3$ &\includegraphics[width=2cm]{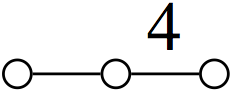}&$F_4$ &\includegraphics[width=2cm]{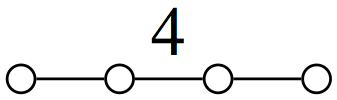} & $E_7^+$ \includegraphics[width=3cm]{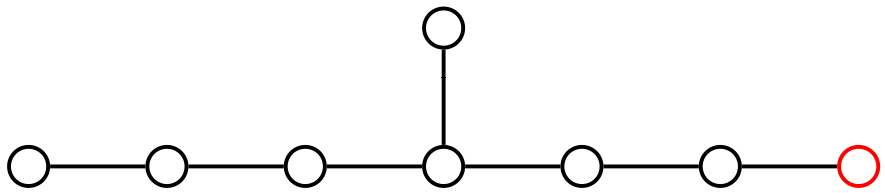}\\
\hline\noalign{\smallskip}
$H_3$ &\includegraphics[width=2cm]{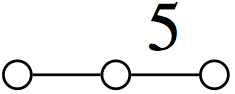}&$H_4$ &\includegraphics[width=2cm]{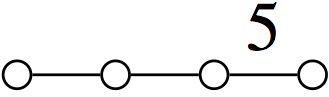} & $E_8^+$ \includegraphics[width=4cm]{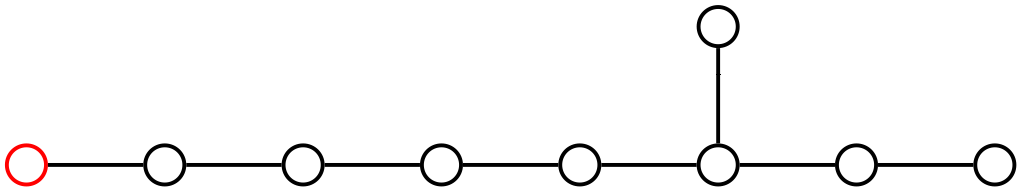}\\
\hline\noalign{\smallskip}
\end{tabular}
\end{table}

Not only is there an abundance of root systems related to the Platonic solids as well as their exceptional nature, but they also have unusual automorphism groups, in that the order of the groups grows as the square of the number of roots. This is also explained via the above spinor construction by means of the following result (which is simply guaranteed by closure of the spinor group under group multiplication, reversal and multiplication by $-1$): 
\begin{theorem}[Spinorial symmetries]\label{HGAsymmetry}
A root system induced via the Clifford spinor construction of a binary polyhedral spinor group $G$ has an automorphism group that trivially contains two factors of the respective spinor group $G$ acting from the left and from the right.
\end{theorem}
This systematises many case-by-case observations on the structure of the automorphism groups \cite{Koca2006F4,Koca2003A4B4F4}. For instance, the automorphism group of the $H_4$ root system is $2I\times 2I$; in the spinor picture, it is not surprising that $2I$ yields both the root system and the two factors in the automorphism group. 

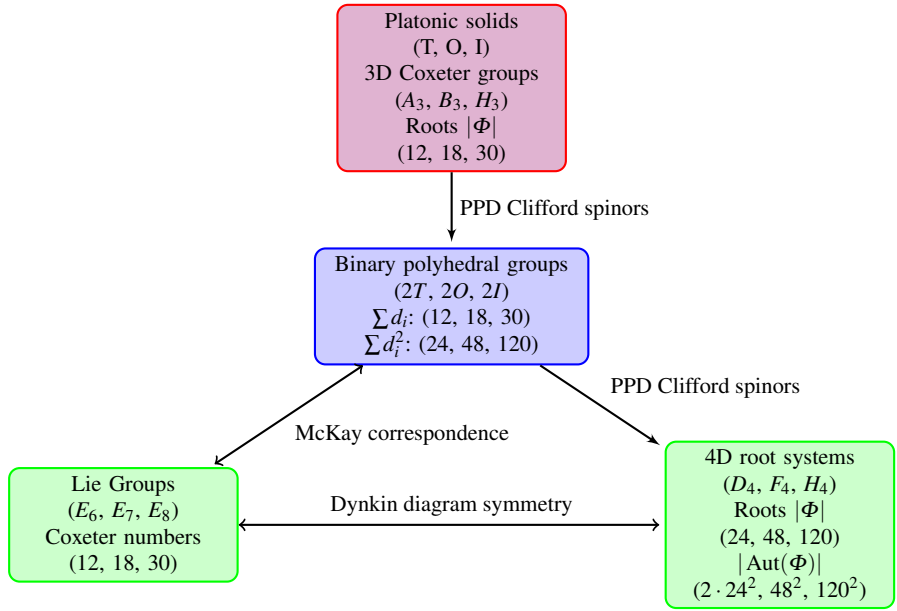
\begin{figure}

		\begin{tikzpicture}
		[scale = 0.5, auto,
		block/.style ={rectangle, draw=blue, thick, fill=blue!20,
		text width=11.5em, text centered, rounded corners,
		minimum height=4em}
		,
		blockr/.style ={rectangle, draw=red, thick, fill=blue!40!red!30!,
		text width=9.5em, text centered, rounded corners,
		minimum height=4em}
		,
		blockg/.style ={rectangle, draw=green, thick, fill=green!20,
		text width=9.5em, text centered, rounded corners,
		minimum height=4em},
		blockw/.style ={rectangle, draw=black, thick,
		text width=9.5em, text centered, rounded corners,
		minimum height=4em}
		,line/.style ={draw, thick, -latex',shorten >=2pt},
		]
		\hspace{-1cm}
		\matrix [column sep=10mm,row sep=10mm]
		{

		&  &\node [blockr] (YR2) {Platonic solids\\(T, O, I)\\3D Coxeter groups\\ ($A_3$, $B_3$, $H_3$)\\Roots $|\Phi|$\\($12$, $18$, $30$)}; \\

		&  &\node [block] (MBYR2) {Binary polyhedral groups \\ ($2T$, $2O$, $2I$)\\ $\sum d_i$:  ($12$, $18$, $30$)\\ $\sum d_i^2$: ($24$, $48$, $120$)}; & \\

		& \node [blockg] (AG2YR1) {Lie Groups\\  ($E_6$, $E_7$, $E_8$)\\ Coxeter numbers \\ ($12$, $18$, $30$)}; &&\node [blockg] (AG2YR3) {4D root systems\\($D_4$, $F_4$, $H_4$)\\ Roots $|\Phi|$\\  ($24$, $48$, $120$)\\$|\Aut(\Phi)|$\\  ($2\cdot24^2$, $48^2$,   $120^2$)};\\

		\\
		};
		\begin{scope}[every path/.style=line]


			\path (YR2)  to  node [midway]   {PPD Clifford spinors} (MBYR2);
				\path (MBYR2) [<->]  to  node [midway]   {McKay correspondence} (AG2YR1);	
			\path (MBYR2)  to  node [midway]   {PPD Clifford spinors} (AG2YR3);	
			\path (AG2YR1) [<->] to  node [midway]   {Dynkin diagram symmetry} (AG2YR3);

		\end{scope}
		\end{tikzpicture}

%
\caption[$E_6^+$]{Web of connections putting the original McKay correspondence and trinities into a wider context. The connection between the sum of the dimensions of the irreducible representations $d_i$ of the binary  polyhedral groups and the Coxeter number of the Lie algebras actually goes all the way back to the number of roots in the 3D root systems $(12, 18, 30)$ -- these then induce the binary polyhedral groups (linked to McKay) and the 4D root systems via the Clifford spinor construction. }
\label{figMcKay}
\end{figure}

We noted earlier that the binary polyhedral spinor groups and the $ADE$-type affine Lie algebras are connected via the McKay correspondence \cite{Mckay1980graphs}, for instance the binary polyhedral groups $(2T, 2O, 2I)$ and the Lie algebras $(E_6, E_7, E_8)$ -- for these $(12, 18, 30)$ is both the  Coxeter number of the respective Lie algebra and the sum of the dimensions of the irreducible representation of the polyhedral group.

However, the connection between $(A_3, B_3, H_3)$ and $(E_6, E_7, E_8)$ via Clifford spinors does not seem to be known. In particular, we note that $(12, 18, 30)$ is exactly the number of roots $\Phi$ in the 3D root systems $(A_3, B_3, H_3)$, which feeds through to the binary polyhedral groups and via the McKay correspondence all the way to the affine Lie algebras. Our construction therefore makes deep connections between trinities and puts the McKay correspondence into a wider framework, as shown in Table \ref{tab:4} and Figure \ref{figMcKay}. It is also striking that the affine Lie algebras and the 4D root systems trinities have identical Dynkin diagram symmetries: $D_4$ and $E_6^+$ have triality $S_3$, $F_4$ and $E_7^+$ have an $S_2$ symmetry and $H_4$ and $E_8^+$ only have $S_1$, but are intimately related as explained in Section \ref{sec_spin}.

\section{Group and representation theory with Clifford multivectors}\label{sec_CGA}

The usual picture of orthogonal transformations on an $n$-dimensional vector space is via $n\times n$ matrices acting on vectors, immediately establishing connections with representations. The above spinor techniques are somewhat unusual; however, it is easy to construct representations in this picture. Orthogonal transformations in the $2^n$-dimensional Clifford algebra leave the original $n$-dimensional vector space invariant; one can therefore consider various representation matrices acting on different subspaces of the Clifford algebra such as -- but not limited to -- the original vector space.

\begin{table}
\caption{Character table for the icosahedral group $I$.}
\label{tab_char_I}       
%
%
%
\begin{tabular}{p{1.8cm}p{1.8cm}p{1.8cm}p{1.8cm}p{1.8cm}p{1.0cm}}
\hline\noalign{\smallskip}
$I$&$1$&$20C_3$&$15C_2$&$12C_5$&$12C_5^2$\\
\noalign{\smallskip}\svhline\noalign{\smallskip}
$1$&$1$&$1$&$1$&$1$&$1$\\
$3$&$3$&$0$&$-1$&$\tau$&$\sigma$\\
$\bar{3}$&$3$&$0$&$-1$&$\sigma$&$\tau$\\
$4$&$4$&$1$&$0$&$-1$&$-1$\\
$5$&$5$&$-1$&$1$&$0$&$0$\\
\hline\noalign{\smallskip}
\end{tabular}
\end{table}

The scalar subspace of the Clifford algebra is one-dimensional. Double-sided action of spinors $R$ gives the trivial representation, since $\tilde{R}1R=\tilde{R}R=1$, and likewise pinors $A$ give the parity. 

The double-sided action of spinors $R$ on a vector $x$ in the $n$-dimensional vector space gives an $n\times n$-dimensional representation, which is just the usual $SO(n)$ representation in terms of rotation matrices;  similar applies to pinors and $O(n)$. For instance, for the spinor examples considered above, $\alpha_1\alpha_2$ and $\alpha_2\alpha_3$, the corresponding rotation matrices with the spinors acting as $\tilde{R}xR$ are 
$$\frac{1}{2}	\begin{pmatrix} \tau&\tau-1&-1 \\ 1-\tau&-1&-\tau\\ -1&\tau&1-\tau \end{pmatrix} \text{ and } 	\frac{1}{2} \begin{pmatrix} \tau&1-\tau&-1 \\ 1-\tau&1&-\tau\\ 1&\tau&\tau-1  \end{pmatrix}.$$
The characters $\chi(g)$ are obviously $0$ and $\tau$ in these cases, and correspond to two different conjugacy classes of the icosahedral group, as shown in Table \ref{tab_char_I}. For a general spinor $R=a_0+a_1Ie_1+a_2Ie_2+a_3Ie_3$ one has 
$$\frac{1}{2}	\begin{pmatrix} a_0^2+a_1^2-a_2^2-a_3^2&-2a_0a_3+2a_1a_2&2a_0a_2+2a_1a_3 \\2a_0a_3+2a_1a_2&a_0^2-a_1^2+a_2^2-a_3^2&-2a_0a_1+2a_2a_3\\ -2a_0a_2+2a_1a_3&2a_0a_1+2a_2a_3&a_0^2-a_1^2-a_2^2+a_3^2 \end{pmatrix} \text{ and } 3a_0^2-a_1^2-a_2^2-a_3^2, $$
so one can read off the character directly from the spinor components. 
If the spinors were acting as $Rx\tilde{R}$ (or alternatively one considers $\alpha_2\alpha_1$ and $\alpha_3\alpha_2$), then the rotation matrices would be given by 
$$\frac{1}{2} \begin{pmatrix} \tau&1-\tau&-1 \\ \tau-1&-1&\tau\\ -1&-\tau&1-\tau \end{pmatrix} \text{ and } 	
  \frac{1}{2} \begin{pmatrix} \tau&1-\tau&1 \\ 1-\tau&1&\tau\\ -1&-\tau&\tau-1  \end{pmatrix},$$
with the same characters as before.
One sees that the first example are 3-fold rotations and the second are 5-fold rotations; swapping the action of the spinor changes to the contragredient representation: if $R$ is in $12C_5$ then $\tilde{R}$ is in $12C_5^2$, and they both have the same character $\tau$ -- i.e. one exchanges the $3$ and the $\bar{3}$ by this operation. 

However, rather than restricting oneself to the  $n$-dimensional vector space, one can also define representations by  $2^n\times  2^n$-matrices acting on the whole Clifford algebra. Likewise, one can define $2^{(n-1)}\times  2^{(n-1)}$-dimensional representations as acting on the even subalgebra. For instance, for the spinors considered above which have components in  ($1$, $e_1e_2$, , $e_2e_3$, $e_3e_1$), multiplication with another spinor will reshuffle these components  ($1$, $e_1e_2$, $e_2e_3$, $e_3e_1$); this reshuffling can therefore be described by a $4\times 4$-matrix. For the examples used above, for the two specific spinors $\alpha_1\alpha_2$ and $\alpha_2\alpha_3$ multiplying a generic spinor $R=a_4+a_1Ie_1+a_2Ie_2+a_3Ie_3$ from the left reshuffles the components $(a_1, a_2, a_3, a_4)$ with the matrices given as
$$\frac{1}{2} \begin{pmatrix} -1&\tau-1&0& -\tau\\ 1-\tau&-1&-\tau&0\\ 0&\tau&-1&\tau-1 & \\ \tau&0&1-\tau &-1 \end{pmatrix} \text{ and } 	
 \frac{1}{2} \begin{pmatrix} -\tau&0&1-\tau&-1 \\ 0&-\tau&-1&\tau-1&\\ \tau-1&1&-\tau &0 \\ 1&1-\tau&0 & -\tau\end{pmatrix},$$
 with characters $-2$ and $-2\tau$. Of course there is a corresponding set of matrices where the spinor acts by right multiplication. 

These matrices are part of a representation of the icosahedral group of the so-called quaternionic type. Other polyhedral groups also have representations of quaternionic type, which seems to be regarded as deeply significant yet appears to be poorly understood. Since the 3D unit spinors  ($1$, $e_1e_2$, $e_2e_3$, $e_3e_1$) are isomorphic to the quaternion algebra, the appearance of quaternionic representations  is not very surprising from a Clifford algebra point of view. In fact, the above construction constructs the representations of quaternionic type in a uniform way, for any of the polyhedral groups (though irreducibility is a separate issue). The existence of these representations is therefore linked to the existence of the Clifford algebras and the structure of the Spin groups. These representations are therefore also much clearer in the Clifford algebra framework. 

One can easily verify the quaternionic nature of the above representation and the corresponding cases for the other polyhedral groups. 
Representations of quaternionic type $\chi$ are characterised by $||\chi||^2=\frac{1}{|G|}\sum_{g\in G}{|\chi(g)|^2}=4$; with real and complex type representations having $1$ and $2$ on the right-hand side, respectively. It is straightforward to calculate the  corresponding $120$ $4\times 4$ matrices and confirm that indeed $||\chi||^2=480/120=4$, in analogy with the two computational examples above. In fact it is easily shown that the representation matrix that belongs to a general spinor $R=b_4+b_1Ie_1+b_2Ie_2+b_3Ie_3$ is given by 
$$\begin{pmatrix} b_4&b_3&-b_2& b_1\\ -b_3&b_4&b_1&b_2\\ b_2&-b_1&b_4&b_3 \\ -b_1&-b_2&-b_3 &b_4 \end{pmatrix} \text{ and } 	
 \chi=4b_4,$$
such that the character is just given by four times the scalar component of the spinor. 

The more general ways of constructing representations outlined above hold in any dimension, and because of the characterisation of Clifford algebras as matrix algebras over $(\mathbb{R},\mathbb{C},\mathbb{H})$, one expects these to yield a mixture of representations of real, complex and quaternionic type.

\section{Conclusion}\label{sec_concl}

In this paper, we have discussed a Clifford algebra framework for certain discrete groups, based on the simple prescription for performing reflections in Clifford algebra. In fact, the Clifford algebra framework is more natural, as the existence of a quadratic form on the vector space considered in the context of root systems means that the corresponding Clifford algebra is always implicit. The reflection symmetries are the building blocks for many discrete symmetries that are interesting for applications also in mathematical physics and biology. However, the framework itself has led to new insights in pure mathematics, as regards the interplay of the geometry of dimensions three and four, Trinities and the McKay correspondence, as well as group and representation theory. It remains to explore the implications of these results in pure and applied mathematics; this has been begun in pure mathematics in \cite{Dechant2012CoxGA,  Dechant2013Platonic}, but there could also be interesting consequences in high energy physics, where 4D root systems are ubiquitious in String Theory, M-Theory and Grand Unified Theories. In particular, recently I was also able to derive the $E_8$ root system consisting of $240$ roots via an analogous Clifford construction of a double cover of the full icosahedral group $H_3$ of order $120$  \cite{Dechant2015Birth,Dechant2015AGACSE} and using a certain reduced inner product due to Wilson \cite{Wilson1986E8}. 


\begin{acknowledgement}
	I would like to thank Reidun Twarock, Anne Taormina, David Hestenes, Anthony Lasenby, John Stillwell, Jozef Siran, Robert Wilson 
	and Ben Fairbairn. 
\end{acknowledgement}


\end{document}